\documentclass[11pt]{amsart}
\usepackage{amsmath,amsthm,amssymb}
\usepackage{stmaryrd}
\usepackage[textsize=scriptsize,obeyFinal]{todonotes}
\usepackage{tikz-cd}

\newtheorem{theorem}{Theorem}[section]
\newtheorem*{theorem*}{Theorem}
\newtheorem{lemma}[theorem]{Lemma}
\newtheorem*{lemma*}{Lemma}
\newtheorem{proposition}[theorem]{Proposition}
\newtheorem*{proposition*}{Proposition}

\newtheorem*{corollary*}{Corollary}

\theoremstyle{definition}
\newtheorem{definition}[theorem]{Definition}
\newtheorem*{definition*}{Definition}
\newtheorem{example}[theorem]{Example}
\newtheorem*{example*}{Example}
\newtheorem{ques}[theorem]{Question}
\newtheorem*{ques*}{Question}

\newtheorem*{claim*}{Claim}
\newtheorem{conjecture}[theorem]{Conjecture}
\newtheorem*{conjecture*}{Conjecture}
\newtheorem{remark}[theorem]{Remark}
\newtheorem*{remark*}{Remark}

\newtheorem*{question*}{Question}

\usepackage{bm}
\newcommand{\ii}{\item}

\newcommand{\half}{\frac12}

\newcommand{\NN}{\mathbb N}

\newcommand{\ol}{\overline}

\usepackage{asymptote}
\begin{asydef}
	dotfactor *= 1.5;
\end{asydef}

\usepackage{amsaddr}
\usepackage{mathrsfs}

\title[A Family of Posets with Small Balance Constant]
{A Family of Partially Ordered Sets with Small Balance Constant}
\date{\today}
\author{Evan Chen}
\address{Department of Mathematics, Massachusetts Institute of Technology}
\email{evanchen@mit.edu}


\newcommand{\PP}{\mathcal P}
\DeclareMathOperator{\pr}{pr}

\begin{document}

\begin{abstract}
	Given a finite poset $\mathcal P$
	and two distinct elements $x$ and $y$,
	we let $\operatorname{pr}_{\mathcal P}(x \prec y)$
	denote the fraction of linear extensions of $\mathcal P$
	in which $x$ precedes $y$.
	The balance constant $\delta(\mathcal P)$
	of $\mathcal P$ is then defined by
	\[
		\delta(\mathcal P) = \max_{x \neq y \in \mathcal P}
		\min \left\{ \operatorname{pr}_{\mathcal P}(x \prec y),
		\operatorname{pr}_{\mathcal P}(y \prec x) \right\}.
	\]
	The $1/3$-$2/3$ conjecture asserts that $\delta(\mathcal P) \ge \frac13$
	whenever $\mathcal P$ is not a chain,
	but except from certain trivial examples
	it is not known when equality occurs,
	or even if balance constants can approach $1/3$.

	In this paper we make some progress on the conjecture
	by exhibiting a sequence of posets with balance constants
	approaching $\frac{1}{32}(93-\sqrt{6697}) \approx 0.3488999$,
	answering a question of Brightwell.
	These provide smaller balance constants
	than any other known nontrivial family.
\end{abstract}

\maketitle

\section{Introduction}
\subsection{Definitions}
Given a finite poset (partially ordered set) $\PP = (\PP, \le)$, and 
distinct elements $x,y \in \PP$,
we let $\pr_\PP(x \prec y)$ denote the proportion of linear extensions
of $\PP$ in which $x$ precedes $y$.
In particular, $\pr_\PP(x \prec y) + \pr_\PP(y \prec x) = 1$,
and if $x \le y$ in $\PP$ then $\pr_\PP(x \prec y) = 1$.

The \emph{balance constant} $\delta(\PP)$ is then defined by
\[
	\delta(\PP) = \max_{x \neq y \in \PP}
	\min \left\{ \pr_\PP(x \prec y), \pr_\PP(y \prec x) \right\}.
\]
(If $\PP$ consists of one element, we let $\delta(\PP) = 0$.)
Thus $\delta(\PP) \in [0, \half]$ for any finite poset $\PP$;
in fact $\delta(\PP) = 0$ exactly when $\PP$ is a chain.

\subsection{The $1/3$-$2/3$ Conjecture}
The main conjecture about balance constants is
the famous $1/3$-$2/3$ conjecture.
\begin{conjecture}
	[$1/3$-$2/3$ conjecture]
	If $\PP$ is a finite poset which is not a chain,
	then $\delta(\PP) \ge \frac 13$.
\end{conjecture}

This conjecture was first proposed in 1968 by Kislitsyn \cite{Kis68},
then again by Fredman in 1976 \cite{Fre76} and Linial \cite{Lin84}.
All three were motivated by the information-theoretic context
of comparison sorting,
but the problem is of course interesting in its own right.

The $1/3$-$2/3$ conjecture has been studied extensively.
The best bound which has been shown for all posets
is due to Brightwell, Felsner, and Trotter \cite{BFT95} in 1995,
who showed that
\[ \delta(\PP) \ge \frac{5-\sqrt5}{10} \approx 0.276393 \]
whenever $\PP$ is not a chain.
This improved a result of Kahn and Saks \cite{KS84} in 1984
which showed the weaker estimate $\delta(\PP) \ge \frac{3}{11} \approx 0.272727$.

While still open for general partially ordered sets,
the conjecture has been proven for several other families
of partially ordered sets,
for example posets of width $2$ by Linial \cite{Lin84}
and posets of height $2$ by Trotter, Gehrlein, Fishburn \cite{TGF92}.
In 2006, Peczarski described an even stronger conjecture,
the so-called ``gold partition conjecture'',
which implies the $1/3$-$2/3$ conjecture;
Peczarski proved this conjecture for posets
with at most $11$ elements \cite{Pec06},
and later for $6$-thin posets \cite{Pec08}.

An extensive survey on the problem is given by Brightwell \cite{BWsurvey},
which describes it as ``one of the major open problems
in the combinatorial theory of partial orders''.

\subsection{Posets with small balance constant}
The following example shows that the constant $1/3$
in best possible.
\begin{example}
	\label{ex:equality}
	Consider the poset $T$ with three elements $\{a,b,c\}$
	with the single relation $a \le b$ (shown in Figure~\ref{fig:T}).
	Then $\delta(T) = \frac 13$.
	\begin{figure}[ht]
		\centering
		\includegraphics{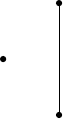}
		\caption{The poset $T$ with $\delta(T) = 1/3$.}
		\label{fig:T}
	\end{figure}
	It follows that linear sums of $T$ and the singleton poset
	have balance constant $1/3$.
\end{example}

However, other than this example,
little is known about the possible sets of balance constants.
For example, it is not known whether there are any other
posets which achieve a balance constant of exactly $1/3$,
other than those in the example above.
It is not even known whether balance constants
can be arbitrarily close to $1/3$.

In Brightwell's survey \cite[Section 4]{BWsurvey},
an example of partially ordered set with $A$
with $\delta(A) = \frac{16}{45} \approx 0.355556$ is given.
Brightwell also gives a family of partially ordered sets
with balance constant approaching
$\frac{7-\sqrt{17}}{8} \approx 0.359612$,
and asks the following two questions.
\begin{ques}
	Is there a poset with balance constant
	between $\delta(T) = \frac13$ and $\delta(A) = \frac{16}{45}$?
\end{ques}
\begin{ques}
	Is $\frac{7-\sqrt{17}}{8} \approx 0.359612$
	the lowest possible limit point other than $1/3$?
\end{ques}

Olson and Sagan \cite{Sagan} resolve the first question
by finding a poset $C$ with
\[ \delta(C) = \frac{37}{106} \approx 0.34905660 \]
which to the author's knowledge
is the smallest balance constant exceeding $1/3$ 
which appears in the literature.
This poset is shown in Figure~\ref{fig:C}.
\begin{figure}[ht]
	\centering
	\includegraphics{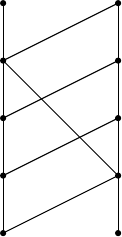}
	\caption{The poset $C$ from \cite[Figure 13]{Sagan}, with $\delta(C) = 37/106$.}
	\label{fig:C}
\end{figure}

The aim of this paper is to answer both questions
with a certain infinite family of partially ordered sets.
We will prove the following theorem.
\begin{theorem}
	There exists a sequence of posets whose balance constants approach
	\[ \kappa = \frac{1}{32} \left( 93 - \sqrt{6697} \right)
		\approx 0.34889999. \]
\end{theorem}

\subsection{Roadmap}
The rest of the paper is divided as follows.
In Section~\ref{sec:setup} we introduce the main players in our proof,
and introduce the notation which we will need for the construction.
Section~\ref{sec:enum} then provides explicit formulas
for the number of linear extensions of our family of posets,
and finally in Section~\ref{sec:horror} we compile these results together
to prove the main theorem.

\subsection*{Acknowledgments}
This research was funded by NSF grant 1659047,
as part of the 2017 Duluth Research Experience for Undergraduates (REU).
The author thanks Joe Gallian for supervising the research,
and for suggesting the problem.
The author is also grateful to Joe Gallian 
for his comments on drafts of the paper.

\section{Setup}
\label{sec:setup}
\begin{definition}
	Throughout the paper let
	$\kappa = \frac{1}{32} \left( 93 - \sqrt{6697} \right)$.
\end{definition}

We first define a ``master poset''
from which our construction will derive.
\begin{definition}
	Let $\PP_\infty$ denote the partially ordered set
	whose elements consist of two infinite $\NN$-indexed chains
	\begin{align*}
		a_1 < a_2 < a_3 < & \cdots \\
		b_1 < b_2 < b_3 < & \cdots
	\end{align*}
	together with the additional covering relations that
	\begin{itemize}
		\ii $a_i \le b_{i+1}$ whenever $i \equiv 1,2,3,4 \pmod 5$, and
		\ii $b_j \le a_{j+2}$ whenever $j \equiv 0,2,4 \pmod 5$.
	\end{itemize}
\end{definition}

All our constructions will be obtained by taking
the bottom-most elements of either chain.
\begin{definition}
	For positive integers $m$ and $n$
	we let $\PP(m,n)$ denote the sub-poset $\PP_\infty$
	induced by taking the elements $\{a_1, \dots, a_m, b_1, \dots, b_n\}$.
\end{definition}
The example $\PP(15, 15)$ is shown in Figure~\ref{fig:15}.

\begin{figure}[ht]
	\centering
	\includegraphics{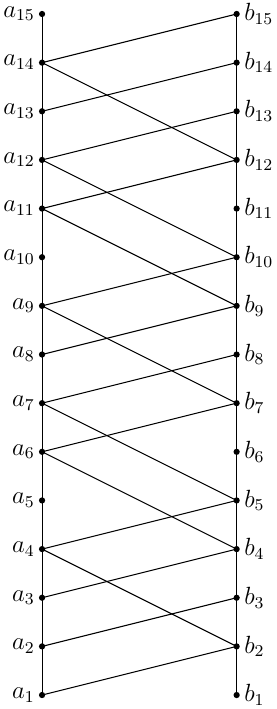}
	\caption{A picture of $\PP(15,15)$.}
	\label{fig:15}
\end{figure}

Our main result is the following.
\begin{theorem}
	As $k \to \infty$,
	\[ \delta\left( \PP(5k, 5k) \right) \to \kappa. \]
\end{theorem}

To approach this result, we introduce further notation.
\begin{definition}
	Let $E(m,n)$ denote the number of
	linear extensions of $\PP(m,n)$.
	For convenience we let $E(0,n) = E(m,0) = 1$ for
	positive integers $m$ and $n$,
	but we leave $E(0,0)$ undefined.
\end{definition}
Then $E(m,n)$ may be computed recursively in the following way.
\begin{proposition}
	\label{prop:down}
	For positive integers $m$ and $n$, we have
	\[ E(m,n) = \begin{cases}
			E(m-1, n) & a_m > b_n \\
			E(m, n-1) & a_m < b_n \\
			E(m-1,n) + E(m,n-1) & \text{otherwise}.
		\end{cases} \]
\end{proposition}
\begin{proof}
	In a linear extension of $\PP(m,n)$,
	either $a_m$ or $b_n$ must be the maximal element,
	and so the recursion follows by considering cases on this.
\end{proof}

According to Proposition~\ref{prop:down},
the interesting cases are those for which $\PP(m,n)$ has no maximal element.
To this end, we introduce the following terminology.
\begin{definition}
	We say the pair $(m,n)$ of positive integers is \emph{admissible}
	if $\PP(m,n)$ has no maximal element.
\end{definition}
One can in fact characterize all the admissible pairs exactly.
We obtain, essentially by definition,
the following characterization.
\begin{lemma}
	The pair $(m,n)$ is admissible if and only if
	it is one of the following forms:
	\begin{enumerate}
		\ii $m=5k+4$ and $n \in \{5k+3, 5k+4\}$.
		\ii $m=5k+3$ and $n \in \{5k+1, 5k+2, 5k+3\}$.
		\ii $m=5k+2$ and $n \in \{5k+1, 5k+2\}$.
		\ii $m=5k+1$ and $n \in \{5k, 5k+1\}$.
		\ii $m=5k$ and $n \in \{5k-2, 5k-1, 5k, 5k+1\}$.
	\end{enumerate}
\end{lemma}

\begin{remark}
	Note that this means that $(m,n)$ is admissible
	if and only if $|m-n| \le 5$,
	in which case only the residues $m \pmod 5$, $n \pmod 5$
	are relevant.
	In particular, if $(m,n)$ is admissible
	then so is $(m+5, n+5)$.
\end{remark}

\section{Enumeration}
\label{sec:enum}
We now proceed to give explicitly compute $E(m,n)$ using induction.
Several base cases are needed for this proof; we do not address these here,
but simply record the results in Appendix~\ref{app:smallcase}.

In order to make this possible, we make the following observation.

\begin{lemma}
	If $(m,n)$ is admissible then
	\[ E(m+10, n+10) = 164 E(m+5, n+5) - 27 E(m,n). \]
\end{lemma}
\begin{proof}
	One may verify this manually for $m,n \le 15$.
	Once this is done the result follows by induction on $m+n$
	owing to Proposition~\ref{prop:down}.
\end{proof}

This implies that the values of $E(m,n)$ satisfy a linear recurrence.
Thus it makes sense to introduce the roots of the corresponding
characteristic polynomial.
\begin{definition}
	Throughout this paper, let
	\begin{align*}
		\theta &= 82 + \sqrt{6697} \approx 163.8352 \\
		\ol{\theta} &= 82 - \sqrt{6697} \approx 0.1648
	\end{align*}
	be the two roots of the polynomial $t^2-164t+27$.
\end{definition}

Then, a direct computation
using the results of Appendix~\ref{app:smallcase}
allows us to compute explicit closed forms:
\begin{proposition}
	\label{prop:explicits}
	\allowdisplaybreaks
	We have the following twelve closed forms.
	\begin{align}
		E(5k+4, 5k+4) &=
			\frac{3025}{2\sqrt{6697}}\left(\theta^k-\ol{\theta}^k\right)
			+\frac{37}{2} \left( \theta^k + \ol{\theta}^k \right) \\
		E(5k+4, 5k+3) &=
			\frac{1883}{2\sqrt{6697}}\left(\theta^k-\ol{\theta}^k\right)
			+\frac{23}{2} \left( \theta^k + \ol{\theta}^k \right) \\
		E(5k+3, 5k+3) &=
			\frac{571}{\sqrt{6697}}\left(\theta^k-\ol{\theta}^k\right)
			+ 7 \left( \theta^k + \ol{\theta}^k \right) \\
		E(5k+3, 5k+2) &=
			\frac{741}{2\sqrt{6697}}\left(\theta^k-\ol{\theta}^k\right)
			+\frac{9}{2} \left( \theta^k + \ol{\theta}^k \right) \\
		E(5k+3, 5k+1) &=
			\frac{170}{\sqrt{6697}}\left(\theta^k-\ol{\theta}^k\right)
			+ 2 \left( \theta^k + \ol{\theta}^k \right) \\
		E(5k+2, 5k+2) &=
			\frac{401}{2\sqrt{6697}}\left(\theta^k-\ol{\theta}^k\right)
			+\frac{5}{2} \left( \theta^k + \ol{\theta}^k \right) \\
		E(5k+2, 5k+1) &=
			\frac{247}{2\sqrt{6697}}\left(\theta^k-\ol{\theta}^k\right)
			+\frac{3}{2} \left( \theta^k + \ol{\theta}^k \right) \\
		E(5k+1, 5k+1) &=
			\frac{77}{\sqrt{6697}}\left(\theta^k-\ol{\theta}^k\right)
			+ \left( \theta^k + \ol{\theta}^k \right) \\
		E(5k+1, 5k) &=
			\frac{93}{2\sqrt{6697}}\left(\theta^k-\ol{\theta}^k\right)
			+\frac{1}{2} \left( \theta^k + \ol{\theta}^k \right) \\
		E(5k, 5k+1) &=
			\frac{61}{2\sqrt{6697}}\left(\theta^k-\ol{\theta}^k\right)
			+\frac{1}{2} \left( \theta^k + \ol{\theta}^k \right) \\
		E(5k, 5k) &=
			\frac{77}{3\sqrt{6697}}\left(\theta^k-\ol{\theta}^k\right)
			+\frac{1}{3} \left( \theta^k + \ol{\theta}^k \right) \\
		E(5k, 5k-1) &=
			\frac{125}{6\sqrt{6697}}\left(\theta^k-\ol{\theta}^k\right)
			+\frac{1}{6} \left( \theta^k + \ol{\theta}^k \right) \\
		E(5k, 5k-2) &=
			\frac{16}{\sqrt{6697}}\left(\theta^k-\ol{\theta}^k\right).
	\end{align}
\end{proposition}

\section{Computing the balance constant}
\label{sec:horror}
Throughout this section, we fix the poset $\PP = \PP(5k, 5k)$.
With Proposition~\ref{prop:explicits},
we now turn to estimating the balance constant of $\PP$.
The point is that Proposition~\ref{prop:explicits}
essentially lets us compute
$\pr_{\PP}(a_i \prec b_j)$ for any $i$ and $j$ already.
For example, we already have that
\begin{align*}
	\pr_{\PP} (a_{5k} \prec b_{5k})
	&= \frac{E(5k, 5k-1)}{E(5k, 5k)} \\
	&= \frac{\frac16 \left( \frac{125}{\sqrt{6697}}+1 \right)}%
	{\frac13 \left( \frac{77}{\sqrt{6697}} +1 \right)}
		\qquad\text{as } k \to \infty \\
	&= \frac{1}{32} \left( \sqrt{6697}-61 \right) \\
	&= 1 - \kappa.
\end{align*}
Thus our goal is to show the following.
\begin{proposition}
	\label{prop:goal}
	For any $i,j \in \left\{ 1, \dots, 5k \right\}$ we have
	\[ \min \left\{ \pr_{\PP} (a_i \prec b_j),
		\pr_{\PP} (a_j \prec b_i) \right\} \le \kappa. \]
\end{proposition}

\begin{proof}
We give the full proof of Proposition~\ref{prop:goal} only in
the case where $i \equiv 1 \pmod 5$,
since the other cases can be resolved in exactly the same fashion.
For notational convenience, we set
\begin{align*}
	i &= 5t+1 \\
	s &= k-t.
\end{align*}
We will assume $t > 0$,
since the $t = 1$ case corresponds to $\pr_{\PP} (a_1 \prec b_1)$
which is in any case equal to $\pr_{\PP} (b_{5n} \prec a_{5n})$ by symmetry.

Consider a linear extension $\prec$ of $\PP$ then.
Since $b_{5t-1} \le a_{5t+1} \le b_{5t+2}$,
we have three distinct possibilities.

\subsection{Case $b_{5t-1} \prec a_{5t+1} \prec b_{5t}$}
Then if we add the relation
$b_{5t-1} \le a_{5t+1} \le b_{5t}$ to $\PP$,
the resulting poset is isomorphic to
the linear sum of $\PP(5t, 5t-1)$
and an inverted copy of $\PP(5s+1, 5s-1)$.
An example with $(k,t) = (3,1)$ is shown in Figure~\ref{fig:cut}.

\begin{figure}[ht]
	\centering
	\includegraphics{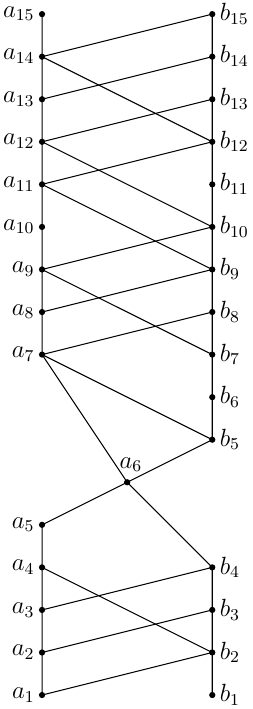}
	\caption{Adding the condition $b_4 \le a_5 \le b_5$ to $\PP(15,15)$.}
	\label{fig:cut}
\end{figure}

The number of linear extensions in this case is then
\begin{align*}
	E(5t, 5t-1) E\left( 5s+1, 5s-1 \right)
	&= E(5t, 5t-1) E\left( 5s, 5s-1 \right) \\
	&= \left[ \frac16 \left( \frac{125}{\sqrt{6997}}+1 \right)
		+ o(1) \right]^2 \theta^k \\
	&= \left[ \frac{11161+125\sqrt{6697}}{12 \cdot 6697} + o(1) \right] \theta^k \\
	&\approx (0.17745 + o(1)) \theta^k.
\end{align*}

\subsection{Case $b_{5t} \prec a_{5t+1} \prec b_{5t+1}$}
Then if we add the relation
$b_{5t} \le a_{5t+1} \le b_{5t+1}$ to $\PP$,
the resulting poset is isomorphic to
the linear sum of $\PP(5t, 5t)$
and an inverted copy of $\PP(5s, 5s-1)$.

The number of linear extensions in this case is then
$E(5t, 5t) E(5s, 5s-1)$, which equals
\begin{align*}
	&\quad E(5t,5t) E(5s,5s-1) \\
	&= \left[ \frac13 \left( \frac{77}{\sqrt{6997}}+1 \right) + o(1) \right]
	\left[ \frac16 \left( \frac{125}{\sqrt{6997}}+1 \right) + o(1) \right]
	\theta^k \\
	&= \left[ \frac{8161 + 101\sqrt{6697}}{9 \cdot 6697} + o(1) \right] \theta^k \\
	&\approx (0.27253 + o(1)) \theta^k.
\end{align*}

\subsection{Case $b_{5t+1} \prec a_{5t+1} \prec b_{5t+2}$}
Then if we add the relation
$b_{5t} \le a_{5t+1} \le b_{5t+1}$ to $\PP$,
the resulting poset is isomorphic to
the linear sum of $\PP(5t, 5t+1)$
and an inverted copy of $\PP(5s-1, 5s-1)$.

Thus the number of linear extensions
in this case is equal to
\begin{align*}
	&\quad E(5t, 5t+1) E(5s-1, 5s-1) \\
	&= \left[ \frac12 \left( \frac{61}{\sqrt{6997}}+1 \right) + o(1) \right]
	\left[ \frac12 \left( \frac{3025}{\sqrt{6997}}+37 \right) + o(1) \right] \theta^{k-1} \\
	&= \left[ \frac{1411+15\sqrt{6697}}{2 \cdot 6697} + o(1) \right] \theta^k \\
	&\approx \left[ 0.19699 + o(1) \right] \theta^k.
\end{align*}

\subsection{Collating the cases}
On the other hand, the total number of linear extension of $\PP$ is
\begin{align*} E(5t, 5t)
	&= \left[ \frac13\left( \frac{77}{\sqrt{6697}}+1 \right)
		+ o(1) \right] \theta^k \\
	&\approx (0.64697 + o(1)) \theta^k.
\end{align*}
So, division gives
\begin{align*}
	\pr_\PP\left( b_{5t-1} \prec a_{5t+1} \prec b_{5t} \right)
	&= \frac16\left( -\frac{29}{\sqrt{6697}} + 2 \right) + o(1) \\
		&\approx 0.27427+o(1) \\
	\pr_\PP\left( b_{5t} \prec a_{5t+1} \prec b_{5t+1} \right)
		&= \frac16 \left( \frac{125}{\sqrt{6697}}+1 \right) + o(1) \\
		&\approx 0.42124 + o(1) \\
	\pr_\PP\left( b_{5t+1} \prec a_{5t+1} \prec b_{5t+2} \right)
		&= \frac12 \left( \frac{-32}{\sqrt{6697}}+1 \right) + o(1) \\
		&\approx 0.30449 + o(1).
\end{align*}
It follows that
\[ \min \left( \pr_{\PP}(a_{5t+1} \prec b_j),
	\pr_{\PP}(b_j \prec a_{5t+1}) \right)
	< \frac 13 < \kappa \]
for $j \in \{5t, 5t+1\}$, assuming $t > 0$.
Hence it holds for all $j$,
since for $j \notin \{5t, 5t+1\}$ the left-hand side vanishes.

This completes the proof of Proposition~\ref{prop:goal}
when $i \equiv 1 \pmod 5$;
the other four cases are analogous.
\end{proof}

\appendix
\section{Examples of values}
\label{app:smallcase}

The following table lists the values of $E(m,n)$
for $\max(m,n) \le 15$
(except for $E(0,0)$ undefined).
The pairs $(m,n)$ which are admissible are bolded.

\[
\begin{array}{r|rrrrr rrrrr r}
n = & 0 & 1 & 2 & 3 & 4 & 5 & 6 & 7 & 8 & 9 & 10 \\ \hline
m=0 &   & \mathbf{1} & 1 & 1 & 1 & 1 & 1 & 1 & 1 & 1 & 1 \\
m=1 & \mathbf{1} & \mathbf{2} & 2 & 2 & 2 & 2 & 2 & 2 & 2 & 2 & 2 \\
m=2 & 1 & \mathbf{3} & \mathbf{5} & 5 & 5 & 5 & 5 & 5 & 5 & 5 & 5 \\
m=3 & 1 & \mathbf{4} & \mathbf{9} & \mathbf{14} & 14 & 14 & 14 & 14 & 14 & 14 & 14 \\
m=4 & 1 & 4 & 9 & \mathbf{23} & \mathbf{37} & 37 & 37 & 37 & 37 & 37 & 37 \\
m=5 & 1 & 4 & 9 & \mathbf{32} & \mathbf{69} & \mathbf{106} & \mathbf{143} & 143 & 143 & 143 & 143 \\
m=6 & 1 & 4 & 9 & 32 & 69 & \mathbf{175} & \mathbf{318} & 318 & 318 & 318 & 318 \\
m=7 & 1 & 4 & 9 & 32 & 69 & 175 & \mathbf{493} & \mathbf{811} & 811 & 811 & 811 \\
m=8 & 1 & 4 & 9 & 32 & 69 & 175 & \mathbf{668} & \mathbf{1479} & \mathbf{2290} & 2290 & 2290 \\
m=9 & 1 & 4 & 9 & 32 & 69 & 175 & 668 & 1479 & \mathbf{3769} & \mathbf{6059} & 6059 \\
m=10 & 1 & 4 & 9 & 32 & 69 & 175 & 668 & 1479 & \mathbf{5248} & \mathbf{11307} & \mathbf{17366} \\
m=11 & 1 & 4 & 9 & 32 & 69 & 175 & 668 & 1479 & 5248 & 11307 & \mathbf{28673} \\
m=12 & 1 & 4 & 9 & 32 & 69 & 175 & 668 & 1479 & 5248 & 11307 & 28673 \\
m=13 & 1 & 4 & 9 & 32 & 69 & 175 & 668 & 1479 & 5248 & 11307 & 28673 \\
m=14 & 1 & 4 & 9 & 32 & 69 & 175 & 668 & 1479 & 5248 & 11307 & 28673 \\
m=15 & 1 & 4 & 9 & 32 & 69 & 175 & 668 & 1479 & 5248 & 11307 & 28673
\end{array}
\]

\[
\begin{array}{r|rrrrr}
n = & 11 & 12 & 13 & 14 & 15 \\ \hline
m=0 & 1 & 1 & 1 & 1 & 1 \\
m=1 & 2 & 2 & 2 & 2 & 2 \\
m=2 & 5 & 5 & 5 & 5 & 5 \\
m=3 & 14 & 14 & 14 & 14 & 14 \\
m=4 & 37 & 37 & 37 & 37 & 37 \\
m=5 & 328 & 365 & 402 & 439 & 476 \\
m=6 & 318 & 318 & 318 & 318 & 318 \\
m=7 & 811 & 811 & 811 & 811 & 811 \\
m=8 & 2290 & 2290 & 2290 & 2290 & 2290 \\
m=9 & 6059 & 6059 & 6059 & 6059 & 6059 \\
m=10 & \mathbf{23425} & 29484 & 35543 & 41602 & 47661 \\
m=11 & \mathbf{52098} & 52098 & 52098 & 52098 & 52098 \\
m=12 & \mathbf{80771} & \mathbf{132869} & 132869 & 132869 & 132869 \\
m=13 & \mathbf{109444} & \mathbf{242313} & \mathbf{375182} & 375182 & 375182 \\
m=14 & 138117 & 242313 & \mathbf{617495} & \mathbf{992677} & 992677 \\
m=15 & 166790 & 242313 & 859808 & \mathbf{1852485} & \mathbf{2845162}
\end{array}
\]

\bibliographystyle{hplain}
\bibliography{refs}

\end{document}